\documentclass[11pt,a4paper]{amsart}
\usepackage{iftex}
\ifluatex
    \usepackage{luatexja}
\fi
\ifpdftex
    \usepackage[whole]{bxcjkjatype}
\fi

\usepackage{amsmath,amssymb,amsthm}
\usepackage[alphabetic,nobysame]{amsrefs}
\usepackage{amsaddr}
\usepackage{tikz,tikz-cd}

\usepackage{mathtools}
\usepackage{hyperref}

\theoremstyle{definition}
\newtheorem{defi}{Definition}[section]
\newtheorem{ex}[defi]{Example}
\newtheorem{rem}[defi]{Remark}
\theoremstyle{plain}
\newtheorem{thm}[defi]{Theorem}
\newtheorem{prop}[defi]{Proposition}
\newtheorem{lem}[defi]{Lemma}
\newtheorem{cor}[defi]{Corollary}
\newtheorem{problem}[defi]{Problem}

\newcommand{\spec}{\operatorname{Spec}}
\newcommand{\Reg}{\operatorname{Reg}}
\newcommand{\Ann}{\operatorname{Ann}}
\newcommand{\ass}{\operatorname{Ass}}

\newcommand{\grade}{\operatorname{grade}}
\newcommand{\pgrade}{\operatorname{p\text{-}grade}}
\newcommand{\rad}{\operatorname{rad}}
\newcommand{\mdepth}{\operatorname{depth}}
\newcommand{\idht}{\operatorname{ht}}

\newcommand{\Tor}{\operatorname{Tor}}

\newcommand{\wgldim}{\operatorname{w\text{.}gl\text{.}dim}}

\newcommand{\coker}{\operatorname{Coker}}

\newcommand{\N}{\mathbb{N}}

\newcommand{\Q}{\mathbb{Q}}
\newcommand{\Z}{\mathbb{Z}}
\newcommand{\ideal}[1]{\mathfrak{#1}}

\newcommand{\mkset}[2]{\left\{#1\mathrel{}\middle|\mathrel{}#2\right\}}

\let\bar\overline

\begin{document}

\title[Polynomial Extensions and Localization of CM Rings]{Polynomial Extensions and Localization of Non-Noetherian Cohen--Macaulay Rings}
\author{Ryoya ANDO}

\address{SKILLUP NeXt Ltd., Chiyoda, Japan}
\email{\url{r_ando@skillupai.com}}
\email{\url{6122701@alumni.tus.ac.jp}}
\date{\today}

\keywords{polynomial grade, Cohen--Macaulay rings, torsion-free modules, Krull primes}
\subjclass[2020]{13H10 (Primary)}

\def\labelenumi{(\theenumi)}
\begin{abstract}
    This paper studies polynomial extensions and localization of HMCM rings, where HMCM means Cohen--Macaulay in the sense of Hamilton--Marley, a notion for non-Noetherian rings. We show that the HMCM property is not preserved either under polynomial extensions or under localization in general. More precisely, we construct an HMCM ring $A$ such that $A[X]$ is not HMCM, and an HMCM ring $B$ with a prime ideal $\mathfrak q$ such that $B_{\mathfrak q}$ is not HMCM. We also prove a positive result: polynomial rings over stably coherent rings of finite weak global dimension are HMCM. In addition, we revisit polynomial grade, give a counterexample to the ``Moreover'' assertion in \cite{Hamilton-Marley2007}*{Proposition~2.7}, and study localization of torsion-free modules via regular saturation and Krull primes.
\end{abstract}
\maketitle

\section{Introduction}\label{sec:intro}

Throughout this paper, all rings are commutative with the identity element.

Homological algebra over non-Noetherian rings has become increasingly important in commutative algebra, especially through its applications to perfectoid ring theory. However, the behavior of regular sequences over non-Noetherian rings is more delicate than in the Noetherian case. In Noetherian ring theory, the set of zero divisors is well controlled by finitely many prime ideals, namely by associated primes. Once the Noetherian or finitely generated hypothesis is removed, this argument breaks down. For example, it may happen that $(0:_M I)=0$ but that $I$ contains no $M$-regular element. Thus, in order to study notions such as the Cohen--Macaulay property and regularity over non-Noetherian rings, it is not sufficient to use the classical notions of grade and regular sequences without modification. In Section~\ref{sec:background}, we first review generalizations of regular sequences and definitions of non-Noetherian Cohen--Macaulay rings.


In Section~\ref{sec:pgrade}, we discuss polynomial grade. This notion was introduced by Hochster \cite{Hochster1974} and plays an important role in the definition, due to Hamilton--Marley, of Cohen--Macaulay rings over non-Noetherian rings, which we call HMCM rings in this paper. Roughly speaking, polynomial grade provides a more stable invariant than the classical grade by measuring grade after adjoining finitely many polynomial variables. In Theorem~\ref{thm:pgrade-ff-sup}, we prove that, for a finitely generated ideal $I$ and an $A$-module $M$, the polynomial grade can be written as $\pgrade_I M=\sup_B \grade_{IB}(M\otimes_A B)$, where the supremum on the right-hand side is taken over all faithfully flat $A$-algebras $B$. This equality was stated without proof in \cite{Hamilton-Marley2007}. We also show, by means of a counterexample, that the ``Moreover'' assertion in \cite{Hamilton-Marley2007}*{Proposition~2.7} is not valid in general; see Remark~\ref{remark:hm-counterexample}.


Polynomial extensions and localization are two basic operations in commutative algebra, especially in the study of homological properties. In the Noetherian setting, the Cohen--Macaulay property behaves well with respect to both operations: a ring $A$ is Cohen--Macaulay if and only if $A[X]$ is Cohen--Macaulay, and the Cohen--Macaulay property is local. For HMCM rings, however, the corresponding permanence properties are more subtle. It is known that $A[X]$ being HMCM implies that $A$ is HMCM (\cite{Hamilton-Marley2007}*{Corollary~4.6}). It is also known that if $A_{\mathfrak{m}}$ is HMCM for every maximal ideal $\mathfrak{m}$, then $A$ is HMCM (\cite{Hamilton-Marley2007}*{Proposition~4.7}). In this paper, we show that the converses suggested by the Noetherian theory fail in general: HMCM rings are not stable under polynomial extensions, and the HMCM property is not local in general.

In Section~\ref{sec:poly-HMCM}, we study the behavior of HMCM rings under polynomial extensions. We construct an example of an HMCM ring $A$ such that $A[X]$ is not HMCM; see Corollary~\ref{cor:A is HMCM but A[X] is not so}. This shows that the HMCM property is not preserved under polynomial extensions in general. We also prove a positive result: if $A$ is stably coherent and has finite weak global dimension, then $A[X_1,\dots,X_n]$ is locally HMCM, and hence HMCM, for every $n\geq 0$.

In Section~\ref{sec:hmcm-localization}, we turn to localization of HMCM rings. We construct an HMCM ring $B$ and a prime ideal $\mathfrak q$ such that $B_{\mathfrak q}$ is not HMCM. Thus, the HMCM property is not local in general. The same construction also yields a torsion-free module whose localization is not torsion-free, suggesting that the appearance of new regular elements after localization is a basic obstruction.

Motivated by this observation, in Section~\ref{sec:loctf} we study localization of torsion-free modules independently of the HMCM property. We introduce the property $\mathrm{LocTF}$ and characterize it in terms of regular saturation and Krull primes. These results isolate the obstruction appearing in Example~\ref{ex:hmcm-ring-localizes-to-bad-idealization} and may be useful in studying the localization problem under additional finiteness assumptions, such as coherence.

Thus, in general, the HMCM property is preserved by neither polynomial extensions nor localization.

\section{Generalizations of regular sequences and non-Noetherian Cohen--Macaulay rings}\label{sec:background}

As mentioned in the Introduction, homological algebra over non-Noetherian rings has become increasingly important in commutative algebra, especially through its applications to perfectoid ring theory. One of the most striking achievements in this direction is the solution of the big Cohen--Macaulay conjecture.

\begin{thm}[Big Cohen--Macaulay conjecture, \cites{Andre2018-1,Andre2018-2}]\label{thm:bigCM}
    Let $(A, \ideal{m})$ be a Noetherian local ring.
    Then there exists an $A$-algebra $B$ such that $B \neq \ideal{m} B$ and a system of parameters of $A$ forms a regular sequence on $B$.
    Such a $B$ is called a big Cohen--Macaulay algebra over $A$.
\end{thm}

The big Cohen--Macaulay conjecture had long been known in equal characteristic, and the mixed characteristic case was solved using perfectoid algebras. In the equal characteristic theory, the Frobenius map plays a central role. The following theorem of Kunz is a basic example of the role played by the Frobenius map; a mixed characteristic analogue was proved by Bhatt--Iyengar--Ma \cite{Bhatt-Iyengar-Ma2019}.




\begin{thm}[\cite{Kunz1976}]\label{thm:Kunz}
    Let $A$ be a Noetherian ring of characteristic $p>0$. Then the following conditions are equivalent.
    \begin{enumerate}
        \item $A$ is regular.
        \item The Frobenius map $F\colon A \to A$, $a \mapsto a^p$, is flat.
    \end{enumerate}
\end{thm}

\begin{thm}[\cite{Bhatt-Iyengar-Ma2019}*{Theorem~4.7}]\label{thm:BIM-kunz}
    Let $A$ be a Noetherian ring such that $p \in \rad(A)$. Then the following conditions are equivalent.
    \begin{enumerate}
        \item $A$ is regular.
        \item There exists a faithfully flat perfectoid $A$-algebra $B$.
    \end{enumerate}
\end{thm}

Non-Noetherian algebras are used in an essential way in these results. In particular, when $\dim A \geq 1$, the absolute integral closure $A^+$ is not Noetherian. It is therefore desirable to develop a theory of homological algebra that can also be applied to such rings.


In Noetherian ring theory, Cohen--Macaulay rings have been studied extensively as a particularly well-behaved class of rings. Recall that a Noetherian local ring $A$ is called Cohen--Macaulay, or CM for short, if it satisfies the following equivalent conditions.

\begin{enumerate}
  \item $\dim A = \mdepth A$. Here $\mdepth A$ is equal to the maximal length of a regular sequence contained in the maximal ideal of $A$.
  \item Every system of parameters is a regular sequence.
  \item For every proper ideal $I$, one has $\idht I = \grade I$.
\end{enumerate}

However, a naive extension of these notions to non-Noetherian rings does not work well. Indeed, every valuation ring $V$ that is not a field satisfies $\mdepth V=1$, whereas valuation rings of dimension at least $2$ do exist and are necessarily non-Noetherian.


On the other hand, note that every valuation ring is a coherent regular ring. Here coherent means that every finitely generated ideal is finitely presented.

\begin{defi}[\cite{Bertin1971}]
   A ring is called regular if every finitely generated ideal has finite projective dimension.
\end{defi}

This notion is a generalization of classical, that is, Noetherian, regular rings, and Noetherian regular rings are Cohen--Macaulay. Thus it is natural to regard valuation rings, which are not necessarily Noetherian, as Cohen--Macaulay as well. Hence it is not natural to generalize Cohen--Macaulayness to non-Noetherian rings on the basis of the condition $\dim A=\mdepth A$. We are therefore led to seek a generalization to non-Noetherian rings satisfying the following conditions.

\begin{enumerate}
  \item It agrees with the classical definition in the Noetherian case.
  \item Coherent regular rings are CM.
  \item $A$ is CM if and only if $A[X]$ is CM.
  \item $A$ is CM if and only if $A_P$ is CM for every $P\in\spec A$.
\end{enumerate}



Among the conditions above, (3) concerns the behavior of regular sequences and grade under polynomial extensions, while (4) concerns the behavior of regularity under localization. We now review several kinds of sequences that arise naturally as generalizations of regular sequences, especially those that frequently appear in the context of extending Cohen--Macaulayness to non-Noetherian rings.


In the big Cohen--Macaulay conjecture (Theorem~\ref{thm:bigCM}) and in the result of \cite{Bhatt-Iyengar-Ma2019} (Theorem~\ref{thm:BIM4.13}), the focus is also on the behavior of systems of parameters, in particular on whether they are regular. From this point of view, the definition proposed by Hamilton--Marley \cite{Hamilton-Marley2007} is particularly relevant. It is based on a generalization of systems of parameters using weakly proregular sequences. This definition satisfies conditions (1) and (2), as well as the implications that $A[X]$ being HMCM implies that $A$ is HMCM, and that $A$ is HMCM whenever $A_P$ is HMCM for every $P\in\spec A$. The converse implications had remained open; Sections 4 and 5 show that both fail in general.


\begin{defi}[\cite{Schenzel2003}]
  Let $A$ be a ring and let $\underline{a}=a_1,\dots,a_r\in A$.
  We say that $\underline{a}$ is a \textbf{weakly proregular sequence} if, for every $1\leq i\leq r$ and every $n\geq 0$, there exists $m\geq n$ such that the map $\varphi_{mn}\colon H_i(\underline{a}^m)\to H_i(\underline{a}^n)$ is the zero map.
\end{defi}

Greenlees--May \cite{Greenlees-May1992} defined proregular sequences, and Schenzel subsequently introduced the notion of weakly proregular sequences. Every regular sequence is proregular, and every proregular sequence is weakly proregular \cite{Schenzel2003}*{Lemma~2.7}.


The following theorem shows that, in the non-Noetherian setting, weakly proregular sequences can be regarded as sequences for which local cohomology is computed by the \v{C}ech complex.

\begin{thm}[\cite{Schenzel2003}*{Theorem~3.2}]\label{thm:Schenzel}
  Let $A$ be a ring, $\underline{a}=a_1,\dots,a_r\in A$ and $I=(a_1,\dots,a_r)$. $\underline{a}$ is a weakly proregular sequence if and only if for any $i$ and $A$-module $M$, $H^i_I(M)\cong\check{H}^i(\underline{a},M)$ functorially on $M$.
\end{thm}
Schenzel proved Theorem~\ref{thm:Schenzel} using the theory of derived categories. For a simpler proof, see \cite{Ando2022}.

For Noetherian rings, \v{C}ech cohomology is naturally isomorphic to local cohomology, and local cohomology can be computed by means of \v{C}ech cohomology or Koszul cohomology. If $A$ is Noetherian, then every sequence $\underline{a}=a_1,\dots,a_r\in A$ is weakly proregular, so Schenzel's theorem extends the corresponding result in the Noetherian case. As an example of a result that computes local cohomology in the non-Noetherian setting, we recall that weakly proregular sequences were used by Bhatt--Iyengar--Ma \cite{Bhatt-Iyengar-Ma2019} to characterize regularity.


\begin{thm}[\cite{Bhatt-Iyengar-Ma2019}*{Theorem 4.13}]\label{thm:BIM4.13}
    Let $(A,\ideal{m},k)$ be an excellent local domain. Then $A$ is regular if any one of the following conditions holds.
    \begin{enumerate}
        \item $A$ has positive characteristic and $\Tor_i^A(A_{\textrm{perf}},k) = 0$ for some $i\geq1$.
        \item $A$ has positive characteristic and $\Tor_i^A(A^+,k) = 0$ for some $i\geq1$.
        \item $A$ has mixed characteristic, $\dim A\leq 3$, and $\Tor_i^A(A^+,k) = 0$ for some $i\geq1$.
    \end{enumerate}
\end{thm}

This theorem is proved along the following lines. When $\underline{a}$ is taken to be a system of parameters of $A$, one shows that, under the relevant hypotheses, it is weakly proregular over $A_{\mathrm{perf}}$ and over $A^+$. This allows one to deduce the vanishing of $\Tor_i(k,k)$ for sufficiently large $i$, and hence to conclude that $A$ is regular.


We now state the aforementioned generalization of systems of parameters using weakly proregular sequences.

\begin{defi}[\cite{Hamilton-Marley2007}*{Definition 3.1, 4.1}]\label{def:HMCM}
    Let $A$ be a ring. A sequence $\underline{a}=a_1,\dots,a_r\in A$ is called a \textbf{parameter sequence} if it satisfies the following conditions.
    \begin{enumerate}
        \item $\underline{a}$ is a weakly proregular sequence.
        \item $\underline{a}A \neq A$.
        \item For every prime ideal $P$ containing $\underline{a}$, one has $\check{H}^r(\underline{a},A)_P \neq 0$.
    \end{enumerate}
    The sequence $\underline{a}$ is called a \textbf{strong parameter sequence} if $a_1,\dots,a_i$ is a parameter sequence for every $i=1,\dots,r$.
    The ring $A$ is called \textbf{Cohen--Macaulay} if every strong parameter sequence is a regular sequence. In this paper, in order to distinguish this notion from the Noetherian one, we say that $A$ is HMCM.
\end{defi}

If $A$ is Noetherian, then the notions of parameter sequences and systems of parameters coincide. Thus this definition is a generalization of the classical definition for Noetherian rings.

Hamilton--Marley \cite{Hamilton-Marley2007} showed that several classes of rings are indeed HMCM. Examples include the absolute integral closure $A^+$ of an excellent domain $A$ of characteristic $p>0$.



\section{Polynomial grade}\label{sec:pgrade}

In their arguments, Hamilton--Marley use polynomial grade, an extension of the classical notion of grade introduced by Hochster \cite{Hochster1974}. In this section, we briefly review polynomial grade and give a complete proof of a relation, stated without proof in \cite{Hamilton-Marley2007}, between polynomial grade and the supremum of classical grades. We also present a counterexample to the assertion of \cite{Hamilton-Marley2007}*{Proposition~2.7}.


\begin{defi}[\cite{Hochster1974}]
    Let $A$ be a ring and let $M$ be an $A$-module. Let
    $\underline{a}=a_1,\dots,a_r\in A$ be a sequence.
    We say that $\underline{a}$ is a weak $M$-sequence if $a_i$ is a non-zero-divisor on $M/(a_1,\dots,a_{i-1})M$ for each $i$.
\end{defi}

If a weak $M$-sequence $\underline{a}$ satisfies $M/\underline{a}M\neq 0$, then it is nothing but a regular sequence. This terminology follows \cite{Bruns-Herzog1997}. Hochster \cite{Hochster1974} calls such a sequence a possibly improper regular sequence on $M$.


For an ideal $I$, let $\grade_I(M)$ denote the supremum of the lengths of weak $M$-sequences contained in $I$, and let $\mdepth_I(M)$ denote the supremum of the lengths of regular $M$-sequences contained in $I$. For grade, the following fact holds in the Noetherian case.


\begin{lem}\label{lem:grade-positive-noetherian}
    Let $A$ be a Noetherian ring, let $I$ be an ideal of $A$, and let $M$ be a finitely generated $A$-module.
    \[\grade_I(M) > 0 \Longleftrightarrow (0:_M I)\coloneq\mkset{x\in M}{Ix = 0}=0.\]
\end{lem}

\begin{proof}
    The only if direction is clear, so we prove the converse. Suppose that $(0:_M I)=0$. Then $I\not\subset P$ for every $P\in\ass M$. Since $A$ is Noetherian and $M$ is finitely generated, $\ass M$ is finite. Hence, by prime avoidance, $I\not\subset \bigcup_{P\in\ass M}P$. Since the right-hand side is the set of all zero-divisors on $M$, there exists an $M$-regular element in $I$.
\end{proof}

However, once the Noetherian condition is removed, it can happen that $(0:_M I)=0$ while the grade is $0$. In other words, although for every nonzero $x\in M$ there exists $a\in I$ such that $ax\neq 0$, it may still be the case that for every $a\in I$ there exists a nonzero $x\in M$ such that $ax=0$. Such examples can be constructed, for instance, by using idealization.


\begin{defi}[\cite{Nagata1962}]\label{def:idealization}
  Let $A$ be a ring and let $M$ be an $A$-module. The direct sum $A\oplus M$ becomes a ring with the following operations:
    \[(a,x)+(b,y) = (a+b,x+y),\]
    \[(a,x)(b,y) = (ab,ay+bx).\]
    We denote this ring by $A\ast M$ and call it the idealization of $M$ over $A$, or the trivial extension of $A$ by $M$.
\end{defi}

An element $(a,x)\in A\ast M$ is regular on $A\ast M$ if and only if $a$ is regular on both $A$ and $M$.

\begin{ex}[\cite{Vasconcelos1971}]\label{ex:vas}
    Let $k$ be a field, and set $A=k[[x,y]]$, $\ideal{m}=(x,y)$, and $M=\bigoplus_{P\in\spec A, \idht P = 1}A/P$. Then, in $A\ast M$, one has $(0:_{A\ast M}\ideal{m}\ast M)=0$, but $\grade_{\ideal{m}\ast M} A\ast M = 0$.
\end{ex}


Now, if Lemma~\ref{lem:grade-positive-noetherian} is extended to polynomial rings, the following fact holds.

\begin{lem}[\cite{Northcott1976}*{Chapter~5, Theorem~7}]\label{lem:Nor thm7}
    Let $A$ be a ring and let $M$ be an $A$-module. For a finitely generated ideal $I=(a_1,\dots,a_r)$, the condition
    $\grade_{IA[X]}(M\otimes_A A[X]) > 0$ is equivalent to $(0:_M I)=0$.
\end{lem}

This leads us to consider the following invariant.

\begin{defi}[\cite{Northcott1976}*{Chapter~5.5}]
    Let $A$ be a ring, let $I$ be an ideal of $A$, and let $M$ be an $A$-module.
    \[\pgrade_I M \coloneq \lim_{n\to\infty}\grade_{IA[X_1,\dots,X_n]}(M[X_1,\dots,X_n]).\]
    We call this the polynomial grade of $M$ with respect to $I$.
\end{defi}

In general, one has
\[
\lim_{n\to\infty}\grade_{IA[X_1,\dots,X_n]}(M[X_1,\dots,X_n])
\leq \sup_B \grade_{IB}(M\otimes_A B),
\]
where $B$ ranges over all faithfully flat $A$-algebras. In \cite{Hamilton-Marley2007}, it is mentioned without proof that equality holds in this formula. In this section, we give a complete proof; see Theorem~\ref{thm:pgrade-ff-sup}.

Moreover, when $(I,M)$ is admissible in the sense of Definition~\ref{def:adm}, Hochster \cite{Hochster1974} takes the right-hand side as the definition of a non-classical grade.



\begin{defi}[\cite{Hochster1974}]\label{def:adm}
    We say that $(I,M)$ is \textbf{admissible} if, for every faithfully flat $A$-algebra $B$, every weak $(M\otimes_A B)$-sequence contained in $IB$ is a regular sequence on $M\otimes_A B$.
\end{defi}

If $IM\neq M$, then $(I,M)$ is admissible. Moreover, when $M$ is finitely generated, the condition $IM\neq M$ is equivalent to the admissibility of $(I,M)$. Thus, if $M$ is finitely generated and $IM=M$, then $(I,M)$ is not admissible.


\begin{prop}[\cite{Hochster1974}*{Section 1, Proposition 2}]\label{prop:Hoc prop2}
    Let $A$ be a ring, let $I$ be an ideal of $A$, and let $M$ be an $A$-module. Let $B$ be a faithfully flat $A$-algebra. If $(I,M)$ is admissible, then there exists $n\geq 0$ such that $\grade_{IB}(M\otimes_A B)\leq\grade_{IA[X_1,\dots,X_n]}(M[X_1,\dots,X_n])$.
\end{prop}

It follows that, when $IM\neq M$, the polynomial grade of $M$ with respect to $I$ is equal to the supremum of $\grade_{IB}(M\otimes_A B)$ over all faithfully flat $A$-algebras $B$.


We now consider the case $IM=M$. If $M$ is finitely generated, then the left-hand side is $\infty$ by Nakayama's lemma, and hence the equality follows.

\begin{prop}
    Let $A$ be a ring, let $I$ be an ideal of $A$, and let $M$ be a finitely generated $A$-module such that $IM=M$. Then $\grade_I M=\infty$. In particular, $\lim\grade(M[X_1,\dots,X_n])=\infty$, and the desired equality holds.
\end{prop}
\begin{proof}
    Since $IM=M$, Nakayama's lemma gives an element $a\in A$ such that $a+1\in I$ and $aM=0$. Then $a+1$ is $M$-regular, and $M/(a+1)M=0$. Hence we can continue by choosing $a+1$ repeatedly, obtaining a sequence $a+1,a+1,\dots$. Thus there are weak $M$-sequences contained in $I$ of arbitrary length, and hence $\grade_I(M)=\infty$.
\end{proof}

Thus the remaining issue is the case where $IM=M$ and $M$ is not finitely generated.


\begin{ex}\label{ex:IM=M}
    Let $A=\Z$, $I=2\Z$, and $M=\mkset{a/2^n+\Z}{a\in \Z,\ n\geq 0}\subset \Q/\Z$. Then $IM=M$, but no element of $I$ can be $M$-regular. Thus $\grade_I M=0$. Hence, when $IM=M$, it is not necessarily true that $\grade_I(M)=\infty$. In this example, since $I$ is principal, one has $\grade_{IB}(M\otimes_A B)=0$ for every faithfully flat $A$-algebra $B$. Consequently, the desired equality holds, with both sides equal to $0$.

\end{ex}

The condition $\grade_I M=0$ does not necessarily imply that $\pgrade_I M=0$. In particular, in Vasconcelos' example (Example~\ref{ex:vas}), Lemma~\ref{lem:Nor thm7} shows that the polynomial grade is positive, although in that example one has $IM\neq M$.


Moreover, if $(I,M)$ is not admissible, then the following holds.

\begin{lem}
    Let $A$ be a ring, let $I$ be an ideal of $A$, and let $M$ be an $A$-module. If $(I,M)$ is not admissible, then
    \[
    \sup_B \grade_{IB}(M\otimes_A B)=\infty,
    \]
    where $B$ ranges over all faithfully flat $A$-algebras.
\end{lem}

\begin{proof}
    There exist a faithfully flat $A$-algebra $B$ and a sequence $\underline{b}$ contained in $IB$ which is a weak $(M\otimes_A B)$-sequence but is not a regular sequence on $M\otimes_A B$. Thus $\underline{b},b_0,b_0,\dots$ is a weak $(M\otimes_A B)$-sequence of arbitrary length.
\end{proof}

We may therefore assume that $(I,M)$ is not admissible. It remains to show that $\lim\grade(M[X_1,\dots,X_n])=\infty$. We first prepare for this.


\begin{lem}[\cite{McCoy1942}]\label{lem:McCoy}
    Let $A$ be a ring and let $M$ be an $A$-module. A polynomial $f\in A[X]$ is a zero-divisor on the $A[X]$-module $M[X]$ if and only if, writing $f=a_0+\cdots+a_nX^n$, there exists a nonzero element $x\in M$ such that $a_i x=0$ for every $i$.
\end{lem}

\begin{proof}   
    The if direction is clear.
    Set
    \[
    m\coloneq\min\mkset{k\in\N}{\text{there exists $0\neq g\in M[X]$ with $\deg g=k$ and $fg=0$}}.
    \]
    Choose $g=x_0+\cdots+x_mX^m\in M[X]$ such that $fg=0$ and $\deg g=m$. Put $h\coloneq a_ng\in M[X]$. Since $fg=0$, we have $a_nx_m=0$, and hence $\deg h<m$. On the other hand, $fh=0$, so by the choice of $m$, we must have $h=0$. Thus $fg=(a_0+\cdots+a_{n-1}X^{n-1})g=0$. If we put $h'\coloneq a_{n-1}g$, then the same argument gives $a_{n-1}x_m=0$, hence $\deg h'<m$, and since $fh'=0$, we must have $h'=0$. Repeating this argument, we see that $x_m$ is annihilated by all the coefficients $a_i$.
\end{proof}

\begin{prop}\label{prop:not-admissible}
    Suppose that $(I,M)$ is not admissible. Then
    \[
    \lim_{n\to\infty}\grade_{IA[X_1,\dots,X_n]}(M[X_1,\dots,X_n])=\infty.
    \]
\end{prop}
\begin{proof}
    By \cite{Hochster1974}*{Section~1, Proposition~3}, there exists a sequence
    $\underline{a}=a_1,\dots,a_r\subset I$ such that
    $H_i(\underline{a},M)=0$
    for every $i$.

    Consider the Koszul complex $K_\bullet(\underline{a},M)$. Its differential in degree $r$ is given by
    \[
        d_r\colon M\to M^r;
        x\mapsto (a_1x,-a_2x,\dots,\pm a_rx).
    \]
    Hence
    \[
        H_r(\underline{a},M)
        =
        \ker d_r
        =
        \mkset{x\in M}{\text{$a_i x=0$ for every $i$}}.
    \]
    Since all Koszul homology groups vanish, we have
    \[
        \mkset{x\in M}{\text{$a_i x=0$ for every $i$}}=0.
    \]
    Therefore, by Lemma~\ref{lem:McCoy}, the element
    $u_1\coloneq \sum_{j=1}^r a_jX_1^j$
    is regular on $M[X_1]$.

    For each $i$, we have
    \[
        H^{A[X_1]}_i(\underline{a},M[X_1])
        =
        H^A_i(\underline{a},M)\otimes_A A[X_1]
        =
        0.
    \]
    Moreover, considering the long exact sequence of Koszul homology induced by the short exact sequence
    \[
    \begin{tikzcd}
        0\arrow[r]
        & {M[X_1]}
        \arrow[r, "u_1\cdot"]
        & {M[X_1]}
        \arrow[r]
        & {M[X_1]/u_1M[X_1]}
        \arrow[r]
        & 0,
    \end{tikzcd}
    \]
    we obtain $H^{A[X_1]}_i
        \bigl(\underline{a},M[X_1]/u_1M[X_1]\bigr)
        =
        0.$

    Now set
    $u_2\coloneq \sum_{j=1}^r a_jX_2^j.$
    The same argument shows that $u_2$ is regular on $M[X_1]/u_1M[X_1]$, and that
    \[
        H^{A[X_1,X_2]}_i
        \bigl(
            \underline{a},
            M[X_1,X_2]/u_1M[X_1,X_2]
        \bigr)
        =
        0.
    \]
    Continuing in this way, for each $n$ we obtain a weak $M[X_1,\dots,X_n]$-sequence $u_1,\dots,u_n
    $
    contained in $IA[X_1,\dots,X_n]$. Hence
    \[
        \lim_{n\to\infty}
        \grade_{IA[X_1,\dots,X_n]}(M[X_1,\dots,X_n])
        =
        \infty.
    \]
\end{proof}

\begin{thm}\label{thm:pgrade-ff-sup}
    Let $A$ be a ring, let $I$ be an ideal of $A$, and let $M$ be an $A$-module. Then
    \[
    \lim_{n\to\infty}\grade_{IA[X_1,\dots,X_n]} M[X_1,\dots,X_n]
    =
    \sup_B \grade_{IB}(M\otimes_A B),
    \]
    where $B$ ranges over all faithfully flat $A$-algebras.
\end{thm}

As mentioned above, this equality is stated without proof in
\cite{Hamilton-Marley2007}, although it may be well known to experts.

\begin{proof}
    If $(I,M)$ is admissible, then the equality follows from Proposition~\ref{prop:Hoc prop2}. If $(I,M)$ is not admissible, then the equality follows from Proposition~\ref{prop:not-admissible}.
\end{proof}

\begin{rem}\label{remark:hm-counterexample}
    We also mention a counterexample to an assertion stated in \cite{Hamilton-Marley2007}.
    According to \cite{Hamilton-Marley2007}*{Proposition~2.7}, if $A$ is a ring, $I=(a_1,\dots,a_r)$ is a finitely generated ideal, and $M$ is an $A$-module, then
    \[\begin{aligned}
        \pgrade_I(M) &= \sup\mkset{k \ge 0}{H_{r-i}(\underline{a}, M) = 0 \text{ for all } i < k} \\
        &= \sup\mkset{k \ge 0}{\check{H}_I^i(M) = 0 \text{ for all } i < k}. 
    \end{aligned}\tag{$\ast$}\]

    The proposition also contains the following ``Moreover'' assertion:
    \[
        \pgrade_I(M)<\infty \quad \Longleftrightarrow \quad IM\neq M.
    \]
    We show that this assertion is not valid in general. When $IM=M$ and $M$ is finitely generated, there is no problem, since Nakayama's lemma gives $\grade_I M=\infty$.
    The problematic case is when $IM=M$ and $M$ is not finitely generated, and Example~\ref{ex:IM=M} gives a counterexample.
    In that example, one has $\pgrade_I(M)=0$. Moreover, direct computations of the Koszul homology and \v{C}ech cohomology show that the two right-hand sides in $(\ast)$ are also equal to $0$. Thus this example does not contradict the identities in $(\ast)$, and the proof given in \cite{Hamilton-Marley2007} appears to establish those identities. What fails is only the additional ``Moreover'' assertion stated without proof.

\end{rem}

\section{Polynomial extensions of HMCM rings}\label{sec:poly-HMCM}

For Cohen--Macaulayness in the sense of Hamilton--Marley, stability under polynomial extensions is not known in general. Indeed, \cite{Hamilton-Marley2007}*{Corollary~4.6} implies that if $A[X]$ is HMCM, then $A$ is HMCM, but it had remained open whether $A[X]$ is HMCM whenever $A$ is HMCM.


In this section, we construct an example of an HMCM ring $A$ such that $A[X]$ is not HMCM; see Corollary~\ref{cor:A is HMCM but A[X] is not so}. If $A$ is Noetherian, then of course $A[X]$ is also Noetherian, and if $A$ is Cohen--Macaulay, then so is $A[X]$. However, even if $A$ is coherent, it is not known in general whether $A[X]$ is coherent. Recall that a ring $A$ is called stably coherent if $A$ is coherent and, for every positive integer $n$, the polynomial ring $A[X_1,\ldots,X_n]$ is coherent. Just as there exist coherent rings that are not stably coherent, such as the examples of Soublin and Alfonsi \cite{Glaz2000}*{Example~8}, \cite{Soublin1968}, \cite{Alfonsi1981}, various properties over non-Noetherian rings are not preserved under polynomial extensions in general. The example constructed in this paper belongs to this circle of phenomena.


We also isolate the structure of the proof of the HMCM property for polynomial rings over finite-dimensional Pr\"ufer domains, in particular over valuation rings, given in \cite{Kim-Walker2020}*{Theorem~25}. We show that, if $A$ is stably coherent and has finite weak global dimension, then every polynomial ring $A[X_1,\dots,X_n]$ in finitely many variables is HMCM. Although their statement assumes finite dimensionality, this reformulation shows that this assumption is in fact unnecessary.


In what follows, we denote by $\Reg(A)$ the set of all non-zero-divisors of a ring $A$. We call the localization $Q(A)=\operatorname{Reg}(A)^{-1}A$ the total ring of fractions of $A$. We also say that an element $a\in A$ is a parameter element if it is a parameter sequence when regarded as a sequence of length one.


The following criterion, which is implicit in the proof of \cite{Mahdikhani-Sahandi-Shirmohammadi2018}*{Theorem~4.11}, ensures that an idealization is not HMCM. We use it to construct examples showing that the HMCM property is preserved neither under polynomial extensions (Corollary~\ref{cor:A is HMCM but A[X] is not so}) nor under localization (Corollary~\ref{cor:hmcm-not-local}).


\begin{prop}\label{prop:idealization-obstruction}
  Let $A$ be a ring, and let $M\neq 0$ be an $A$-module. Suppose that there exists an element $a\in\Reg(A)$ such that $a\notin A^\times$ and $aM=0$. Set $B\coloneq A\ast M$. Then $B$ is not HMCM.
\end{prop}

\begin{proof}
  Let $\alpha=(a,0)\in B$.
  We show that $\alpha$ is a parameter element but is not a regular
  element.
  First, for every $0\neq x\in M$, we have
  $\alpha(0,x)=0$.
  Hence $\alpha$ is a zero-divisor on $B$.
  In particular, $\alpha B\neq B$.

  For a single element, weak proregularity is equivalent to the
  following condition:
  \[
    {}^\forall n\geq 1,
    {}^\exists m\geq n,
    \alpha^{m-n}\Ann_B(\alpha^m)=0.
  \]
  Since  $\Ann_B(\alpha^n)=0\ast M$ for every $n\geq 1$, and since $\alpha(0,x)=0$ for every $x\in M$,
  this condition holds.
  Thus $\alpha$ is weakly proregular.
  It remains to prove the non-vanishing of the \v{C}ech cohomology.
  Since $(0,x)/1=\alpha(0,x)/\alpha=0$ in $B_\alpha$ for every $x\in M$, localization at $\alpha$ gives
  an isomorphism $B_\alpha\cong A_a $. 
  Under this identification, the image of the natural map
  $B\longrightarrow B_\alpha$ is $A\subset A_a$.
  Therefore
  \[
    \check H^1(\alpha,B)\cong A_a/A.
  \]

  Let $\mathfrak q\in\spec B$ be a prime ideal containing $\alpha$.
  Then $\mathfrak q=P\ast M$ for some $P\in\spec A$, and $a\in P$.
  Hence

  \[
    \check H^1(\alpha,B)_{\mathfrak q}
    \cong
    (A_a)_P/A_P.
  \]

  We claim that this module is nonzero.
  Indeed, if $1/a=0$ in $(A_a)_P/A_P$, then $1/a\in A_P$.
  Since $a$ is a non-zero-divisor on $A_P$, this would force $a$
  to be a unit in $A_P$, contradicting $a\in PA_P$.
  Thus
  \[
    (A_a)_P/A_P\neq 0.
  \]

  Therefore $\alpha$ is a parameter element.
  However, $\alpha$ is a zero-divisor on $B$, so it is not a regular
  element.
  Hence $B$ is not HMCM.
\end{proof}

Applying Proposition~\ref{prop:idealization-obstruction}, we obtain the following general criterion for a polynomial extension of an idealization to fail to be HMCM.


\begin{thm}\label{thm:poly-ext-not-hmcm}
    Let $A$ be a ring, and let $I=(a_0,\dots,a_d)\subsetneq A$ be a finitely generated ideal such that $\Ann I=0$. Set $M=A/I$ and take the idealization $B=A\ast M$. Then $B[X]$ is not HMCM.
\end{thm}
\begin{proof}
    Let $f=a_0+a_1X+\dots+a_dX^d\in A[X]$. By McCoy's theorem (Lemma~\ref{lem:McCoy}) and the assumption $\Ann I=0$, we have $f\in\Reg(A[X])$. Moreover, since $a_0\in I\subsetneq A$, the constant term $a_0$ is not a unit, and hence $f\notin A[X]^\times$. 

    We have an isomorphism $B[X]\cong A[X]\ast M[X]$. Indeed, it is given by sending $\sum (a_i,x_i)X^i$ to $(\sum a_iX^i,\sum x_iX^i)$. By construction, we have $fM[X]=0$. Therefore, by Proposition \ref{prop:idealization-obstruction}, the ring $B[X]$ is not HMCM.
\end{proof}


To apply this theorem, we construct a ring $A$ admitting a finitely generated ideal $I$ with $\Ann I=0$ such that $A\ast A/I$ is HMCM.

\begin{ex}\label{ex:Ann I = 0 and B is HMCM}
    Let $k$ be a field, and let $D=k[s,t]_{(s,t)}$. Let $\ideal{m}=(s,t)$ be the unique maximal ideal of $D$.
    Set
    \[
        M=\bigoplus_{0\neq g\in\ideal{m}} D/gD.
    \]
    Consider the idealization $A=D\ast M$ and the ideal $I=((s,0),(t,0))A$ of $A$. Then $\Ann I=0$, and $B=A\ast A/I$ has no parameter elements. In particular, $B$ is HMCM.
\end{ex}

\begin{proof}
    First note that, for any $a,b\in D$ and $u,v\in M$, we have
    \[
        (s,0)(a,u)+(t,0)(b,v)=(sa+tb,su+tv).
    \]
    Hence $I=\ideal{m}\ast \ideal{m}M$.

    We next show that $\Ann I=0$. Let $(a,u)\in \Ann I$. Since $(s,0),(t,0)\in I$, we have $as=at=0$. Since $D$ is a domain, it follows that $a=0$. Looking at the second component, we then have $su=tu=0$. Write $u=(\overline{u_g})\in M=\bigoplus_{0\neq g\in\ideal{m}}D/gD$. Then, for each $0\neq g\in\ideal{m}$, we have $su_g\in gD$ and $tu_g\in gD$. Since $D$ is a UFD and $s,t$ are relatively prime, this implies $u_g\in gD$. Hence $u=0$, and therefore $\Ann I=0$.

    We show that $B$ has no parameter elements. Let $\alpha\in B$ be arbitrary, and write
    \[
        \alpha=(a,x)=((d,u),x)\in (D\ast M)\ast A/I.
    \]

    Case 1. Suppose that $d\in D^\times$. Then $a=(d,u)\in A^\times$. Indeed,
    \[
        (d,u)(d^{-1},-d^{-2}u)=(1,0).
    \]
    Hence $\alpha\in B^\times$ as well, and so $\alpha$ is not a parameter element.

    Case 2. Suppose that $d=0$. Then $a=(0,u)$ satisfies $a^2=0$. Hence
    \[
        \alpha^2=(0,2ax)
        \quad\text{and}\quad
        \alpha^3=0.
    \]
    Thus $\alpha$ is nilpotent. Therefore $B_\alpha=0$, and
    \[
        \check{H}^1(\alpha,B)=\operatorname{Coker}(B\to B_\alpha)=0.
    \]
    Hence $\alpha$ is not a parameter element.

    Case 3. Suppose that $d\neq0$ and $d\notin D^\times$. We show that $\alpha$ is not weakly proregular. For a single element, weak proregularity is equivalent to the following condition: for every $n\geq1$, there exists $m\geq n$ such that
    \[
        \alpha^{m-n}\Ann(\alpha^m)=0.
    \]
    We show that this condition fails for $n=2$. Since $D$ is local, $d\in\ideal{m}$. For each $m\geq2$, define $u_m=(\overline{u_{m,g}})\in M$ by setting $u_{m,g}=1$ if $g=d^m$, and $u_{m,g}=0$ if $g\neq d^m$. Put
    \[
        \beta_m=((0,u_m),0)\in B.
    \]
    Write $\alpha=(a,x)$, where $a=(d,u)$. Then
    \[
        \alpha^m=(a^m,ma^{m-1}x),
    \]
    and hence
    \[
        \alpha^m\beta_m
        =
        \bigl(a^m(0,u_m),\,ma^{m-1}(0,u_m)x\bigr).
    \]
    The first component is
    \[
        a^m(0,u_m)=(0,d^m u_m)=0.
    \]
    To see that the second component is also zero, write $x=\overline{(c,v)}\in A/I$ with $(c,v)\in A$. Then
    \[
        ma^{m-1}(0,u_m)x
        =
        m(d^{m-1},(m-1)d^{m-2}u)(0,u_m)x
        =
        m\overline{(0,d^{m-1}cu_m)}.
    \]
    Since $m\geq2$ and $d\in\ideal{m}$, we have $d^{m-1}cu_m\in \ideal{m}M$. Thus
    \[
        (0,d^{m-1}cu_m)\in 0\ast \ideal{m}M\subset \ideal{m}\ast \ideal{m}M=I,
    \]
    so the second component is zero. Hence $\beta_m\in\Ann(\alpha^m)$.

    On the other hand, the first component of $\alpha^{m-2}\beta_m$ is
    \[
        a^{m-2}(0,u_m)=(0,d^{m-2}u_m).
    \]
    This is nonzero. Indeed, if $d^{m-2}u_m=0$, then looking at the component indexed by $d^m$ gives $d^{m-2}\in d^mD$, which is impossible since $D$ is a domain and $d$ is not a unit. Therefore
    $\alpha^{m-2}\beta_m\neq0.$
    Thus $\alpha^{m-2}\Ann(\alpha^m)\neq0$ for every $m\geq2$, and so $\alpha$ is not weakly proregular.

    In all cases, $\alpha$ is not a parameter element. Therefore $B$ has no parameter elements, and hence no nonempty strong parameter sequences. Consequently, $B$ is HMCM.
\end{proof}

\begin{cor}\label{cor:A is HMCM but A[X] is not so}
    Let $k$ be a field, and let $D=k[s,t]_{(s,t)}$. Let $\ideal{m}=(s,t)$ be the unique maximal ideal of $D$.
    Set
    \[
        M=\bigoplus_{0\neq g\in\ideal{m}} D/gD.
    \]
    Let $A=D\ast M$, let $I=((s,0),(t,0))A$ be an ideal of $A$, and consider the idealization $B=A\ast A/I$. Then $B$ is HMCM, but $B[X]$ is not HMCM.
\end{cor}

\begin{proof}
    By Example~\ref{ex:Ann I = 0 and B is HMCM}, we have $\Ann I=0$, and $B$ is HMCM. Hence, by Theorem~\ref{thm:poly-ext-not-hmcm}, the polynomial ring $B[X]$ is not HMCM.
\end{proof}

\begin{rem}
    We now recall a related result on the HMCM property of idealizations. By \cite{Mahdikhani-Sahandi-Shirmohammadi2018}*{Theorem~4.11} and the note immediately following it, if $A$ is an HMCM ring and every $A$-regular sequence is a weakly $M$-regular sequence for an $A$-module $M$, then the idealization $A\ast M$ is HMCM. In Example~\ref{ex:Ann I = 0 and B is HMCM}, one can show that $A=D\ast M$ has no parameter elements, and hence $A$ is HMCM. Moreover, $A$ has no nonunit regular elements, so the latter condition in \cite{Mahdikhani-Sahandi-Shirmohammadi2018}*{Theorem~4.11} is vacuously satisfied. Thus one can also prove from this result that $B=A\ast A/I$ is HMCM. In Example~\ref{ex:Ann I = 0 and B is HMCM}, we prove the stronger statement that $B$ has no parameter elements.

\end{rem}

Thus we have shown that, even if $A$ is HMCM, the polynomial ring $A[X]$ need not be HMCM. We next give a slight generalization of \cite{Kim-Walker2020}*{Theorem~25}. We show that, if $A$ is stably coherent and has finite weak global dimension, then every polynomial ring $A[X_1,\dots,X_n]$ in finitely many variables is HMCM.

In what follows, $\operatorname{fd}_A M$ and $\operatorname{pd}_A M$ denote the flat dimension and the projective dimension of an $A$-module $M$, respectively. We define the weak global dimension of $A$ by
\[
    \wgldim A
    =
    \sup\mkset{\operatorname{fd}_A M}{ M\text{ is an }A\text{-module}}.
\]


\begin{lem}\label{lem:coherent-fd-pd}
    Let $A$ be a coherent ring, and let $M$ be a finitely presented $A$-module.
    If $\operatorname{fd}_A M<\infty$, then $\operatorname{pd}_A M<\infty$.
\end{lem}

\begin{proof}
    We prove the assertion by induction on $d=\operatorname{fd}_A M$.
    If $d=0$, then $M$ is finitely presented and flat, hence projective.

    Suppose that $d>0$. Since $M$ is finitely presented, there exists an exact sequence
    \[
        0\to K\to F\to M\to 0
    \]
    with $F$ a finite free $A$-module. Since $A$ is coherent, the module $K$ is finitely presented. Moreover, $\operatorname{fd}_A K\leq d-1$. Hence, by the induction hypothesis, $\operatorname{pd}_A K<\infty$. Therefore $\operatorname{pd}_A M<\infty$.
\end{proof}

\begin{prop}\label{prop:coherent-finite-wdim-HMCM}
    Let $A$ be a coherent ring, and assume that $\wgldim A<\infty$.
    Then $A$ is a coherent regular ring.
    In particular, $A$ is locally HMCM, and hence $A$ is HMCM.
\end{prop}

\begin{proof}
    Let $I\subset A$ be a finitely generated ideal.
    Since $A$ is coherent, $I$ is a finitely presented $A$-module.
    Moreover, since $\wgldim A<\infty$, we have $\operatorname{fd}_A I<\infty$.
    Hence Lemma~\ref{lem:coherent-fd-pd} implies that $\operatorname{pd}_A I<\infty$.

    Thus every finitely generated ideal of $A$ has finite projective dimension.
    In other words, $A$ is a coherent regular ring.
    By Hamilton--Marley's theorem \cite{Hamilton-Marley2007}*{Theorem~4.8}, every coherent regular ring is locally HMCM.
    Moreover, by \cite{Hamilton-Marley2007}*{Proposition~4.7}, if the localizations at all maximal ideals are HMCM, then the ring itself is HMCM.
    Therefore $A$ is HMCM.
\end{proof}

\begin{thm}\label{thm:stably-coherent-finite-wdim-poly-HMCM}
    Let $A$ be a stably coherent ring and assume that $\wgldim A<\infty$.
    Then, for every $n\geq0$, the polynomial ring $A[X_1,\dots,X_n]$ is locally HMCM.
    In particular, $A[X_1,\dots,X_n]$ is HMCM.
\end{thm}

\begin{proof}
    Set $B=A[X_1,\dots,X_n]$.
    Since $A$ is stably coherent, the ring $B$ is coherent.
    Moreover, by the Hilbert syzygy theorem for weak global dimension \cite{Jensen1966}*{Theorem~2}, we have $ \wgldim B = \wgldim A+n < \infty. $
    Applying Proposition~\ref{prop:coherent-finite-wdim-HMCM} to $B$, we conclude that $B$ is locally HMCM, and hence HMCM.
\end{proof}

\begin{defi}
    Let $A$ be a ring. $A$ is called hereditary if every ideal of $A$ is projective. Also, $A$ is called \textbf{semi-hereditary} if every finitely generated ideal of $A$ is projective.
\end{defi} 

\begin{cor}\label{cor:semihereditary-polynomial-HMCM}
    Let $A$ be a semi-hereditary ring.
    Then, for every $n\geq0$, the polynomial ring
    $A[X_1,\dots,X_n]$ is locally HMCM.
    In particular, $A[X_1,\dots,X_n]$ is HMCM.
\end{cor}

\begin{proof}
    Since every semi-hereditary ring is stably coherent \cite{Glaz1989}*{Corollary~7.3.4} and satisfies $\wgldim A\leq1$, the assertion follows from Theorem~\ref{thm:stably-coherent-finite-wdim-poly-HMCM}.
\end{proof}

We finally note that the ring $B$ in Example~\ref{ex:Ann I = 0 and B is HMCM} is not coherent. Thus the following problem remains open.
\begin{problem}
    Let $A$ be a coherent Cohen--Macaulay ring in the sense of Hamilton--Marley.
    Is $A[X]$ Cohen--Macaulay in the sense of Hamilton--Marley?
\end{problem}

\section{Localization of HMCM rings}\label{sec:hmcm-localization}




We next turn to localization. We construct an example of a ring that is HMCM globally but whose localization admits an element playing the role of $a$ in Proposition~\ref{prop:idealization-obstruction}. This yields a counterexample showing that the HMCM property is not preserved under localization in general; see Corollary~\ref{cor:hmcm-not-local}.

\begin{ex}\label{ex:hmcm-ring-localizes-to-bad-idealization}
    Let $k$ be an infinite field and set
    \[
        D=k[s,t]_{(s,t)}, \ideal{m}=(s,t), P=sD.
    \]
    Then there exist a $D$-module $M$, an ideal $I\subset A=D\ast M$,
    a ring $B=A\ast A/I$, and a prime ideal $\mathfrak q\in \spec B$
    such that $B$ is HMCM and
    \[
        B_{\mathfrak q}\cong D_P\ast D_P/sD_P.
    \]
\end{ex}

\begin{proof}
    For each $0\neq d\in\ideal{m}$ and $n\geq 2$, we construct a
    $D$-module $N_{d,n}$ and an element $u_{d,n}\in N_{d,n}$
    such that
    \[
        d^n u_{d,n}=0,
        d^{n-1}u_{d,n}\in sN_{d,n},
        d^{n-2}u_{d,n}\neq 0,
        (N_{d,n})_P=0.
    \]

    Case 1. First suppose that $d\notin P$. Write $D^2=De_1\oplus De_2$ with
    basis $e_1,e_2$, and define
    \[
        N_{d,n}
        = D^2/(d^n e_1, d^{n-1}e_1-se_2, de_2)D^2.
    \]
    Put $u_{d,n}=\overline{e_1}$. Then the defining relations give $d^n u_{d,n}=0$ and $d^{n-1}u_{d,n}\in sN_{d,n}$.
    We show that $d^{n-2}u_{d,n}\neq 0$. Consider the $D$-linear map
    \[
        D^2\rightarrow D/d^{n-1}D
        ;
        e_1\mapsto 1, e_2\mapsto 0.
    \]
    It annihilates the defining relations of $N_{d,n}$, and hence induces a
    $D$-linear map $N_{d,n}\rightarrow D/d^{n-1}D$.
    The image of $d^{n-2}u_{d,n}$ is the residue class of $d^{n-2}$, which is
    nonzero since $D$ is a domain and $d$ is a nonunit. Thus $d^{n-2}u_{d,n}\neq 0$. Finally, since $d\notin P$, the element $d$ becomes a unit in $D_P$.
    The localized relations $d^n e_1=0$ and $de_2=0$ therefore imply
    $e_1=e_2=0$. Hence $(N_{d,n})_P=0$.
    
    Case 2. Suppose that $d\in P=sD$. Then $d=sf$ for some $f\in D$. 
    Since $D$ is a UFD and $k$ is infinite, we can choose $\lambda\in k$ so
    that $t-\lambda s$ does not divide $d$; indeed, it suffices to factor
    $d$ and avoid the factors appearing in its factorization. Put $r_{d,n}=t-\lambda s$.
    Then $\bar{D}:=D/r_{d,n}D = k[s]_{(s)}$.
    Suppose that $d^{n-2}\in (d^n,r_{d,n})$.
    Then $\overline{d^{n-2}}\in \overline{d^n}\,\overline D$.
    Since $\overline D$ is a domain, it follows that $\overline d$ is a unit.
    On the other hand, since $d\in P$, the image $\overline d$ is not a unit in $\bar{D} = k[s]_{(s)}$.
    This is a contradiction. Therefore $d^{n-2}\notin (d^n,r_{d,n})$. In this case, define
    \[
        N_{d,n}=D/(d^n,r_{d,n})D
    \]
    and put $u_{d,n}=\overline{1}$. Then $d^n u_{d,n}=0$.
    Moreover, since $d=sf$, we have
    \[
        d^{n-1}u_{d,n} = sfd^{n-2}u_{d,n} \in sN_{d,n}.
    \]
    By the choice of $r_{d,n}, d^{n-2}u_{d,n}\neq 0$.
    Finally, since $r_{d,n}\notin P$, the module $N_{d,n}$ vanishes after
    localizing at $P$. Therefore $(N_{d,n})_P=0$.

    Now set
    \[
        M=\bigoplus_{\substack{0\neq d\in\ideal{m}\\ n\geq 2}}
        N_{d,n}.
    \]
    Then $M_P=0$. Let $A=D\ast M$, put $I=(s,0)A$ = $sD\ast sM$, and set
    $B=A\ast A/I$.

    We show that $B$ is HMCM. In fact, $B$ has no parameter element.
    Let $\alpha\in B$, and write
    \[
        \alpha = (a,x) = ((d,u),x) \in (D\ast M)\ast A/I.
    \]

    If $d \in D^\times$ or $d=0$, then, as in Example \ref{ex:Ann I = 0 and B is HMCM}, we have respectively $\alpha\in B^\times$ and $\check{H}^1_{\alpha}(B)=0$. Hence $\alpha$ is not a parameter element.
    
    It remains to consider the case $0\neq d\in\ideal{m}$. We show that
    $\alpha$ is not weakly proregular. 
    Fix $n\geq 2$, and take $u_{d,n}\in N_{d,n}\subset M$ as
    above. Put
    \[
        b_n=(0,u_{d,n})\in A,
        \beta_n=(b_n,0)\in B.
    \]
    Then $a^m b_n=(0,d^m u_{d,n})$
    for all $m\geq 0$. Hence $a^n b_n=0$, and
    \[
        a^{n-1}b_n
        =
        (0,d^{n-1}u_{d,n})
        \in sD\ast sM
        =
        I.
    \]
    Since $\alpha=(a,x)$, we get
    \[
        \alpha^n\beta_n
        =
        (a^n,n a^{n-1}x)(b_n,0)
        =
        (a^n b_n,n a^{n-1}b_nx)
        =
        0.
    \]
    Thus $\beta_n\in\Ann_B(\alpha^n)$. On the other hand, the first
    component of $\alpha^{n-2}\beta_n$ is $a^{n-2}b_n=(0,d^{n-2}u_{d,n})$,
    which is nonzero by construction. Therefore $\alpha^{n-2}\Ann_B(\alpha^n)\neq 0$ for every $n\geq 2$.
    Hence, $\alpha$ is not weakly proregular. 

    In all cases, $\alpha$ is not a parameter element. Therefore $B$ has no parameter element. In particular, $B$ has no
    nonempty strong parameter sequence, so $B$ is HMCM.

    Finally, let $\mathfrak p=P\ast M\in\spec A$ and
    $\mathfrak q=\mathfrak p\ast A/I\in\spec B$. Since $M_P=0$,
    \[
        A_{\mathfrak p}\cong D_P\ast M_P\cong D_P.
    \]
    Moreover, $I_{\mathfrak p}=(s,0)A_{\mathfrak p}$ corresponds to
    $sD_P$, and hence
    \[
        (A/I)_{\mathfrak p}\cong A_{\mathfrak p}/I_{\mathfrak p}
        \cong D_P/sD_P.
    \]
    Therefore
    \[
        B_{\mathfrak q}
        \cong A_{\mathfrak p}\ast (A/I)_{\mathfrak p}
        \cong D_P\ast D_P/sD_P.
    \]
\end{proof}

\begin{cor}\label{cor:hmcm-not-local}
    Let $k$ be an infinite field. Then there exist a commutative ring
    $B$ and a prime ideal $\mathfrak q\in \spec B$ such that $B$ is
    HMCM, but $B_{\mathfrak q}$ is not HMCM.
\end{cor}

\begin{proof}
    Let $B$ and $\mathfrak q\in \spec B$ be as in
    Example~\ref{ex:hmcm-ring-localizes-to-bad-idealization}. Then $B$ is
    HMCM and
    \[
        B_{\mathfrak q}
        \cong
        D_P\ast D_P/sD_P.
    \]
    The ring $D_P$ is a one-dimensional local domain, and $0\neq s\in sD_P$ is a nonunit. Applying Proposition~\ref{prop:idealization-obstruction},
    we see that $D_P \ast D_P/sD_P$
    is not HMCM. Hence $B_{\mathfrak q}$ is not HMCM.
\end{proof}

\begin{rem}
    In the notation of Example~\ref{ex:hmcm-ring-localizes-to-bad-idealization}, the construction also gives a torsion-free $A$-module whose localization is not torsion-free. Indeed, one has $\Reg(A)=A^\times$, and hence every $A$-module, in particular $A/I$, is torsion-free. On the other hand,
    $A_{\mathfrak p}\cong D_P$ and $(A/I)_{\mathfrak p}\cong D_P/sD_P$,
    and the latter module is not torsion-free over $D_P$. We study this failure of torsion-free localization systematically in the next section.
\end{rem}

%
%

\section{Localization of torsion-free modules}\label{sec:loctf}




The example constructed in the previous section shows that the HMCM property is not preserved under localization in general. Nevertheless, the localization problem remains open under additional finiteness assumptions, for instance for coherent rings. To study possible positive results in this setting, it is necessary to understand how regular elements behave under localization. Every regular element remains regular after localization. Conversely, however, a regular element in a localized ring need not arise from a global regular element. This phenomenon can be formulated as a problem concerning the localization of torsion-free modules. In general, the localization of a torsion-free module need not remain torsion-free. In this section, we study this phenomenon independently of the HMCM property and characterize the rings over which torsion-free modules remain torsion-free after localization.

We say that an $A$-module $M$ is torsion-free if, for every $a\in\operatorname{Reg}(A)$, the natural multiplication map $a\colon M\to M$ is injective. In general, even if $M$ is torsion-free over $A$, the localization $M_P$ need not be torsion-free over $A_P$.


\begin{defi}
    A ring $A$ is said to satisfy the property $\mathrm{LocTF}$ if, for every torsion-free $A$-module $M$ and every prime ideal $P\in\operatorname{Spec} A$, the localization $M_P$ is torsion-free over $A_P$.
\end{defi}

In \cite{Ando2026}, the condition $\dim Q(A)=0$ was given as a sufficient condition for a ring to satisfy $\mathrm{LocTF}$; see \cite{Ando2026}*{Proposition~4.5}. However, this condition is not necessary; see \cite{Ando2026}*{Example~4.9}. In this paper, instead of considering the dimension of $Q(A)$, we examine how the set of zero-divisors behaves under localization and give reformulations of $\mathrm{LocTF}$.


We first reformulate $\mathrm{LocTF}$ in terms of ideals. This reformulation is useful because it eliminates the need to test all modules.


\begin{defi}
    For an ideal $I\subset A$, set
    \[
        I^{\mathrm{reg}}
        =
        \mkset{a\in A}{\text{there exists } b\in \operatorname{Reg}(A) \text{ such that } ab\in I}.
    \]
    We call $I^{\mathrm{reg}}$ the regular saturation of $I$.
\end{defi}


\begin{thm}\label{thm:main-equivalences}
Let $A$ be a ring. The following conditions are equivalent.
\begin{enumerate}
    \item $A$ satisfies $\mathrm{LocTF}$.
    \item For every $P\in\operatorname{Spec} A$ and every ideal $I\subset A$, one has $(I^{\mathrm{reg}})_P=(I_P)^{\mathrm{reg}}$.
    \item For every $P\in\operatorname{Spec} A$ and every $\alpha\in \operatorname{Reg}(A_P)$, one has $\alpha A_P\cap \operatorname{Reg}(A)_P\neq\emptyset$.
    \item For every $P\in\operatorname{Spec} A$, the natural homomorphism $Q(A)_P\rightarrow Q(A_P)$ is an isomorphism.
\end{enumerate}
\end{thm}

\begin{proof}
    We first prove the equivalence of (1) and (3). Assume (3). Let $M$ be a torsion-free $A$-module, and let $\alpha\in\operatorname{Reg}(A_P)$ and $\bar{x}=x/s\in M_P$ be such that $\alpha\bar{x}=0$. By (3), there exist $\beta\in A_P$ and $c/t\in\operatorname{Reg}(A)_P$ such that $\alpha\beta=c/t$. Hence $(c/t)\bar{x}=0$. Thus there exists $h\notin P$ such that $hcx=0$. Since $c\in\operatorname{Reg}(A)$ and $M$ is torsion-free over $A$, we have $hx=0$. Hence $\bar{x}=0$. Therefore $M_P$ is torsion-free over $A_P$.

    Conversely, suppose that (3) fails. Then there exists $\alpha=a/s\in\operatorname{Reg}(A_P)$ such that $\alpha A_P\cap \operatorname{Reg}(A)_P=\emptyset$. Let $J$ be the $\operatorname{Reg}(A)_P$-saturation of $\alpha A_P$ in $A_P$:
    \[
        J\coloneq\mkset{\beta\in A_P}{\text{there exists } c/t\in\operatorname{Reg}(A)_P \text{ such that } \beta(c/t)\in\alpha A_P}.
    \]
    Then $A_P/J$ is torsion-free as an $A$-module. Indeed, if $c\in\operatorname{Reg}(A)$ and $c\beta\in J$, then $\beta\in J$ by the definition of $J$. On the other hand, $\alpha(A_P/J)=0$. Moreover, since $\alpha A_P\cap\operatorname{Reg}(A)_P=\emptyset$, we have $1\notin J$, and hence $A_P/J\neq0$. Thus $A_P/J$ is not torsion-free as an $A_P$-module. Since $(A_P/J)_P\cong A_P/J$, the property $\mathrm{LocTF}$ fails.

    Next we prove the equivalence of (2) and (3). For every ideal $I\subset A$, we always have $(I^{\mathrm{reg}})_P\subset (I_P)^{\mathrm{reg}}$. Assume (3), and let $a/s\in (I_P)^{\mathrm{reg}}$. Then there exists $\alpha\in\operatorname{Reg}(A_P)$ such that $(a/s)\alpha\in I_P$. By (3), there exist $\beta\in A_P$ and $b/t\in\operatorname{Reg}(A)_P$ such that $\alpha\beta=b/t$. Hence $ab/st\in I_P$, so there exists $h\notin P$ such that $hab\in I$. Since $b\in\operatorname{Reg}(A)$, this means that $ha\in I^{\mathrm{reg}}$. Therefore $a/s=ha/(hs)\in (I^{\mathrm{reg}})_P$. Thus (2) holds.

    Conversely, assume (2), and let $\alpha=a/s\in\operatorname{Reg}(A_P)$. Since $s$ is a unit in $A_P$, the element $a/1=s\alpha$ also belongs to $\operatorname{Reg}(A_P)$. Hence $(aA_P)^{\mathrm{reg}}=A_P$. Applying (2) to the ideal $I=aA$, we get $A_P=(aA_P)^{\mathrm{reg}}=((aA)^{\mathrm{reg}})_P$. Therefore there exists $h\notin P$ such that $h\in (aA)^{\mathrm{reg}}$. Hence there exists $c\in\operatorname{Reg}(A)$ such that $ch\in aA$. Since $h$ is a unit in $A_P$, we have $c/1\in aA_P=\alpha A_P$. Thus $\alpha A_P\cap\operatorname{Reg}(A)_P\neq\emptyset$, proving (3).

    Finally, we prove the equivalence of (3) and (4). The natural map
    $Q(A)_P \to Q(A_P)$ is always injective, because every element of
    $\Reg(A_P)$ is a non-zero-divisor on $A_P$. Thus it is an isomorphism if and
    only if it is surjective. Surjectivity means that, for every
    $\alpha\in\Reg(A_P)$, the element $1/\alpha\in Q(A_P)$ comes from $Q(A)_P$.
    This is equivalent to saying that $\alpha$ becomes invertible in $Q(A)_P$.
    Since every element of $\Reg(A)_P$ is a non-zero-divisor on $A_P$, this holds
    if and only if $\alpha A_P$ contains an element of $\Reg(A)_P$. This is
    condition (3).
\end{proof}

Next, focusing on condition (4) of Theorem~\ref{thm:main-equivalences}, we reformulate it further in terms of prime ideals.

\begin{defi}
    Set
    \[
        \Sigma(A)=\mkset{P\in\operatorname{Spec} A}{P\cap\operatorname{Reg}(A)=\emptyset}.
    \]
    Thus $\Sigma(A)$ is the set of prime ideals all of whose elements are zero-divisors on $A$, and it can be naturally identified with $\operatorname{Spec} Q(A)$.
\end{defi}

With this notation, Theorem~\ref{thm:main-equivalences} can be reformulated as follows.

\begin{prop}\label{prop:sigma}
    For a ring $A$, the following conditions are equivalent.
    \begin{enumerate}
        \item $A$ satisfies $\mathrm{LocTF}$.
        \item For every $P\in\operatorname{Spec} A$, one has $ \Sigma(A_P) = \mkset{\mathfrak q A_P}{\mathfrak q\in\Sigma(A), \mathfrak q\subset P}$. 
        \item For every $P\in\Sigma(A)$, one has $PA_P\in\Sigma(A_P)$.
    \end{enumerate}
\end{prop}

Without the Noetherian hypothesis, the usual associated primes do not control zero-divisors sufficiently well. We therefore use Krull primes, introduced by Krull \cite{Krull1929}. Krull primes are one generalization of associated primes. As an example of their use, we mention the work of Fuchs--Heinzer--Olberding \cite{Fuchs-Heinzer-Olberding2006}, which studies irreducibility of ideals without assuming Noetherianity. For issues concerning Krull primes and their generalizations, see also \cite{Epstein-Shapiro2014}.


\begin{defi}[\cite{Krull1929}]
    For an $A$-module $M$, a prime ideal $P\in\operatorname{Spec} A$ is called a Krull prime of $M$ if, for every $a\in P$, there exists $x\in M$ such that $a\in\operatorname{Ann}_A(x)\subset P$.
    We denote by $\operatorname{KAss}_A(M)$ the set of all Krull primes of $M$.
\end{defi}


\begin{lem}\label{lem:krull-local}
    For a prime ideal $P\in\operatorname{Spec} A$, the following conditions are equivalent.
    \begin{enumerate}
        \item $P\in\operatorname{KAss}_A(A)$.
        \item $PA_P\cap\operatorname{Reg}(A_P)=\emptyset$.
        \item $PA_P\in\Sigma(A_P)$.
    \end{enumerate}
\end{lem}

\begin{proof}
    The equivalence of (2) and (3) follows from the definition.

    Assume (1). Let $\alpha=a/s\in PA_P$ be arbitrary. We may assume that $a\in P$. By the definition of Krull primes, there exists $b\in A$ such that $ab=0$ and $\operatorname{Ann}_A(b)\subset P$. Then $b/1\neq0$ in $A_P$, and hence $\alpha$ is a zero-divisor on $A_P$.

    Conversely, assume (2). Let $a\in P$ be arbitrary. Since $a/1\notin\operatorname{Reg}(A_P)$, there exists $b/s\neq0$ in $A_P$ such that $(a/1)(b/s)=0$. Hence there exists $h\notin P$ such that $hab=0$. Replacing $b$ by $hb$, we may assume that $ab=0$ and $b/1\neq0$ in $A_P$. If $c\in\operatorname{Ann}_A(b)$, then $c\in P$. Indeed, if $c\notin P$, then $c$ is a unit in $A_P$, and hence $b/1=0$, a contradiction. Therefore $a\in\operatorname{Ann}_A(b)\subset P$.
\end{proof}

\begin{thm}\label{thm:krull-characterization}
    For a ring $A$, the following conditions are equivalent.
    \begin{enumerate}
        \item $A$ satisfies $\mathrm{LocTF}$.
        \item $\Sigma(A)\subset \operatorname{KAss}_A(A)$.
        \item $\Sigma(A)=\operatorname{KAss}_A(A)$.
    \end{enumerate}
\end{thm}

\begin{proof}
    The equivalence of (1) and (2) follows from Proposition~\ref{prop:sigma} and Lemma~\ref{lem:krull-local}. 

    It remains to compare (2) and (3). We always have
    $\operatorname{KAss}_A(A)\subset \Sigma(A)$. Indeed, if
    $P\in\operatorname{KAss}_A(A)$ and $a\in P$, then by definition there exists
    $b\in A$ such that $a\in\operatorname{Ann}_A(b)\subset P$. In particular,
    $a$ is a zero-divisor. Hence $P\cap\operatorname{Reg}(A)=\emptyset$, so
    $P\in\Sigma(A)$.

    Therefore condition (2) is equivalent to
    $\Sigma(A)=\operatorname{KAss}_A(A)$, which is condition (3).
\end{proof}


We next consider the Noetherian case.

\begin{rem}
    In some references, such as \cite{Iroz-Rush1984}, which discuss inclusion relations among various classes of associated-prime-type notions, one finds statements asserting that, over Noetherian rings, Krull primes coincide with the usual associated primes. We note that this is false, as pointed out in \cite{Epstein-Shapiro2014}*{Remark~2.2}.
\end{rem}

In this paper, we prove that they do coincide when $M$ is finitely generated. The statement itself appears in \cite{Epstein-Shapiro2014}, but the proof is omitted there.


\begin{prop}
    Let $A$ be a Noetherian ring, and let $M$ be a finitely generated $A$-module.
    Then $\operatorname{KAss}_A(M)=\operatorname{Ass}_A(M)$.
\end{prop}

\begin{proof}
    The inclusion $\operatorname{Ass}_A(M)\subset \operatorname{KAss}_A(M)$ is clear from the definition. We prove the reverse inclusion. Let $P\in\operatorname{KAss}_A(M)$. Then $PA_P\in\operatorname{KAss}_{A_P}(M_P)$. Moreover, if $PA_P\in\operatorname{Ass}_{A_P}(M_P)$, then $P\in\operatorname{Ass}_A(M)$. Thus we may replace $A$ by $A_P$ and assume that $(A,\mathfrak{m})$ is local and that $\mathfrak{m}\in\operatorname{KAss}_A(M)$.

    We show that $\mathfrak{m}\in\operatorname{Ass}_A(M)$. Since $\mathfrak{m}\in\operatorname{KAss}_A(M)$, every element of $\mathfrak{m}$ is a zero-divisor on $M$. Hence
    \[
        \mathfrak{m}\subset \bigcup_{\ideal{q}\in\operatorname{Ass}_A(M)} \ideal{q}.
    \]
    Since $A$ is Noetherian and $M$ is finitely generated, the set $\operatorname{Ass}_A(M)$ is finite. Therefore, by prime avoidance, there exists $\ideal{q}\in\operatorname{Ass}_A(M)$ such that $\mathfrak{m}\subset \ideal{q}$. Since $A$ is local, every proper prime ideal is contained in $\mathfrak{m}$, and hence $\ideal{q}=\mathfrak{m}$. Thus $\mathfrak{m}\in\operatorname{Ass}_A(M)$.
\end{proof}

Thus Theorem~\ref{thm:krull-characterization} gives the following consequence.

\begin{cor}\label{cor:noetherian}
    Let $A$ be a Noetherian ring. The following conditions are equivalent.
    \begin{enumerate}
        \item $A$ satisfies $\mathrm{LocTF}$.
        \item $\Sigma(A)=\operatorname{Ass}_A(A)$.
        \item For every $P\in\operatorname{Ass}_A(A)$ and every prime ideal $\mathfrak q\subset P$, one has $\mathfrak q\in\operatorname{Ass}_A(A)$. In other words, $\operatorname{Ass}_A(A)$ is closed under generization.
    \end{enumerate}
\end{cor}

\begin{proof}
    In a Noetherian ring, the set of zero-divisors is the union of the associated primes. Hence
    \[
        \Sigma(A)
        =
        \mkset{\mathfrak q\in\operatorname{Spec} A}
        {\mathfrak q\subset P\text{ for some }P\in\operatorname{Ass}_A(A)}.
    \]
    Moreover, for Noetherian rings, we have
    $\operatorname{KAss}_A(A)=\operatorname{Ass}_A(A)$.

    By Theorem~\ref{thm:krull-characterization}, condition (1) is equivalent to
    \[
        \Sigma(A)=\operatorname{KAss}_A(A).
    \]
    Using $\operatorname{KAss}_A(A)=\operatorname{Ass}_A(A)$, this is exactly condition (2).

    It remains to compare (2) and (3). By the displayed description of $\Sigma(A)$, the equality
    $\Sigma(A)=\operatorname{Ass}_A(A)$ holds if and only if every prime ideal contained in an associated prime is again associated. This is precisely condition (3), namely that $\operatorname{Ass}_A(A)$ is closed under generization.
\end{proof}

\begin{cor}\label{cor:ass-min-loctf}
    Let $A$ be a Noetherian ring. If $\operatorname{Ass}_A(A)=\operatorname{Min}(A)$, then $A$ satisfies $\mathrm{LocTF}$. In particular, every Noetherian reduced ring satisfies $\mathrm{LocTF}$.
\end{cor}

\begin{proof}
    Since $\operatorname{Min}(A)$ is closed under generization, the assertion follows from Corollary~\ref{cor:noetherian}.
\end{proof} 

We note that Corollary~\ref{cor:ass-min-loctf} also follows from the following lemma of Epstein--Yao.

\begin{lem}[\cite{Epstein-Yao2012}*{Lemma~3.8}]
    Let $A$ be a Noetherian ring, and let $M$ be an $A$-module. Then $M$ is torsion-free if and only if
    \[
    \bigcup_{P\in\operatorname{Ass}(M)}P
    \subset
    \bigcup_{P\in\operatorname{Ass}(A)}P.
    \]
    In particular, if $A$ has no embedded primes, then $M$ is torsion-free if and only if $\operatorname{Ass}(M)\subset\operatorname{Ass}(A)$. In this case, for every multiplicative subset $S\subset A$, the localization $S^{-1}M$ of a torsion-free $A$-module $M$ is torsion-free over $S^{-1}A$.
\end{lem}

\begin{rem}
    Let $(A,\mathfrak m)$ be a Noetherian local ring such that $\mathfrak m\in\operatorname{Ass}_A(A)$, equivalently, $\operatorname{depth} A=0$. In this case, the property $\mathrm{LocTF}$ imposes a strong restriction. Indeed, if $\mathrm{LocTF}$ holds, then every prime ideal $\mathfrak q\subset\mathfrak m$ must be an associated prime. Hence $\operatorname{Spec} A=\operatorname{Ass}_A(A)$ .
\end{rem}

We collect several examples, including some that have already appeared above.

\begin{ex}
    If $A$ is a Noetherian reduced ring, then $\operatorname{Ass}_A(A)=\operatorname{Min}(A)$. Hence $A$ satisfies $\mathrm{LocTF}$. This also follows from the fact that $\dim Q(A)=0$.
\end{ex}

\begin{ex}[\cite{Ando2026}*{Example~4.9}]
    Let $k$ be a field, and let $A=k[x,y]_{(x,y)}/(x^2,xy)$. Then $\operatorname{Spec} A=\{(x),(x,y)\}$, and both prime ideals are associated primes. Therefore $\operatorname{Ass}_A(A)$ is closed under generization, and hence $A$ satisfies $\mathrm{LocTF}$ by Corollary~\ref{cor:noetherian}. In this example, one has $A=Q(A)$ and $\dim Q(A)=1$.
\end{ex}

\begin{ex}
    Let $A$ be a Noetherian ring. Suppose that $P\in\operatorname{Ass}_A(A)$ and that there exists a prime ideal $\mathfrak q\subsetneq P$ which is not an associated prime. Then, by Corollary~\ref{cor:noetherian}, the ring $A$ does not satisfy $\mathrm{LocTF}$.
\end{ex}

In \cite{Ando2026}, the fact that localizations of torsion-free modules over semi-hereditary rings remain torsion-free was used via the known result that semi-hereditary rings satisfy $\dim Q(A)=0$; see \cite{Glaz1989}*{Corollary~4.2.19}. As we have seen above, $\mathrm{LocTF}$ is weaker than being semi-hereditary, and this result can be generalized to PP rings. For the control of zero-divisors by PP rings, see also \cite{Glaz2002}.

\begin{defi}
    A ring $A$ is called a PP ring if, for every $a\in A$, the principal ideal $aA$ is projective. This is equivalent to saying that, for every $a\in A$, there exists an idempotent element $e\in A$ such that $\Ann_A(a)=eA$.
\end{defi}

\begin{prop}\label{prop:pp}
    Every PP ring satisfies $\mathrm{LocTF}$. Consequently, every semi-hereditary ring satisfies $\mathrm{LocTF}$.
\end{prop}

\begin{proof}
    Let $A$ be a PP ring, and let $P\in\operatorname{Spec} A$. Suppose that $a/s\in A_P$ is a non-zero-divisor. Choose an idempotent element $e\in A$ such that $\operatorname{Ann}_A(a)=eA$. After localization, we have $\operatorname{Ann}_{A_P}(a/s)=eA_P$.
    Since $a/s$ is a non-zero-divisor on $A_P$, it follows that $e/1=0$ in $A_P$.

    Set $b=a+e$. We claim that $b$ is a non-zero-divisor on $A$. Indeed, suppose that $(a+e)c=0$. Multiplying by $e$, and using $ae=0$, we get $ec=0$. Then $(a+e)c=0$ and $ec=0$ imply $ac=0$, so $c\in\operatorname{Ann}_A(a)=eA$. Hence $c=ec=0$. Thus $b\in\operatorname{Reg}(A)$.

    On the other hand, since $e/1=0$ in $A_P$, we have $b/s=a/s$ in $A_P$. Therefore
    \[
        a/s\in (a/s)A_P\cap \operatorname{Reg}(A)_P.
    \]
    Thus condition (3) of Theorem~\ref{thm:main-equivalences} holds. Hence $A$ satisfies $\mathrm{LocTF}$.

    Finally, every semi-hereditary ring is a PP ring, since every finitely generated ideal is projective, and in particular every principal ideal is projective. Therefore every semi-hereditary ring satisfies $\mathrm{LocTF}$.
\end{proof}

\begin{ex}
    Let $(R,\mathfrak m,k)$ be a one-dimensional local domain, and set $A=R\ast k$. Then $A$ satisfies $\mathrm{LocTF}$. Indeed, we have
    \[
        \operatorname{Spec} A=\{(0)\ast k,\ \mathfrak m\ast k\}.
    \]
    Moreover, $\mathfrak m\ast k=\operatorname{Ann}_A((0,1))$, and for every $0\neq a\in\mathfrak m$, one has $(0)\ast k=\operatorname{Ann}_A((a,0))$. Hence $\Sigma(A)=\operatorname{KAss}_A(A)$, and therefore $\mathrm{LocTF}$ holds.
\end{ex}

\begin{rem}
    In the notation of Example~\ref{ex:hmcm-ring-localizes-to-bad-idealization}, one has $\Reg(A)=A^\times$. Hence every $A$-module, and in particular $A/I$, is torsion-free over $A$. On the other hand,
    \[
        A_{\mathfrak p}\cong D_P, 
        (A/I)_{\mathfrak p}\cong D_P/sD_P,
    \]
    and the latter module is not torsion-free over $A_{\mathfrak p}$, since $s$ is regular on $D_P$ but acts as zero on $D_P/sD_P$.
    
    Equivalently, since $Q(A)=A$, the natural map
    \[
        Q(A)_{\mathfrak p}\cong D_P
        \rightarrow
        Q(A_{\mathfrak p})\cong \operatorname{Frac}(D_P)
    \]
    is not an isomorphism. Thus, by Theorem~\ref{thm:main-equivalences}, the ring $A$ does not satisfy $\mathrm{LocTF}$. 
    
    The construction in Example \ref{ex:hmcm-ring-localizes-to-bad-idealization} exhibits a failure of LocTF: after localization, a new regular element appears and acts as zero on the localized module component. This suggests that the localization behavior of torsion-free modules may be relevant to the localization problem for HMCM rings. 
\end{rem}

\subsection*{Acknowledgments}

The characterization of $\mathrm{LocTF}$ in Section~\ref{sec:loctf} is based on an idea from private communication with Ryota Kuroki concerning supplementary information to \cite{Ando2026}. The author is deeply grateful to him.

\end{document}